\documentclass{amsart}

\usepackage[section]{placeins}

\usepackage{amsmath}             
\usepackage{amscd}
\usepackage{amssymb}             
\usepackage{amsfonts}            
\usepackage{latexsym}            
\usepackage{amsthm}              
\usepackage{graphicx}

\usepackage{enumerate}
\usepackage{mathrsfs} 
\usepackage{thmtools}
\usepackage{thm-restate}

\usepackage{hyperref}

\newcommand{\Z}{\mathbb{Z}}

\newcommand{\SL}{\operatorname{SL}}

\newcommand{\SO}{\operatorname{SO}}

\newcommand{\Aut}{\operatorname{Aut}}

\newcommand{\Sym}{\operatorname{Sym}}

\newcommand{\Lie}{\operatorname{Lie}}

\newcommand{\diag}{\operatorname{diag}}

\newtheorem{lause}{Theorem}[section]

\newtheorem{lemma}[lause]{Lemma}

\newtheorem{prop}[lause]{Proposition}

\newtheorem*{lause*}{Theorem}

\theoremstyle{definition}

\theoremstyle{remark}
\newtheorem{remark}[lause]{Remark}
 
\newtheorem*{mot*}{Motivation}
\newtheorem*{acknow*}{Acknowledgements}

\numberwithin{equation}{section}

\usepackage{hyphenat}
\allowdisplaybreaks

\begin{document}

\title[Exotic $2$-local subgroup in $E_7(q)$]{Structure of an exotic $2$-local subgroup in $E_7(q)$}  

\author{Mikko Korhonen}
\address{M. Korhonen, Shenzhen International Center for Mathematics, Southern University of Science and Technology, Shenzhen 518055, Guangdong, P. R. China}
\email{korhonen\_mikko@hotmail.com {\text{\rm(Korhonen)}}}
\thanks{}

\date{\today}

\begin{abstract}
Let $G$ be the finite simple group of Lie type $G = E_7(q)$, where $q$ is an odd prime power. Then $G$ is an index $2$ subgroup of the adjoint group $G_{\operatorname{ad}}$, which is also denoted by $G_{\operatorname{ad}} = \operatorname{Inndiag}(G)$ and known as the group of inner-diagonal automorphisms. It was proven by Cohen--Liebeck--Saxl--Seitz (1992) that there is an elementary abelian $2$-subgroup $E$ of order $4$ in $G_{\operatorname{ad}}$, such that $N_{G_{\operatorname{ad}}}(E)/C_{G_{\operatorname{ad}}}(E) \cong \operatorname{Sym}_3$, and $C_{G_{\operatorname{ad}}}(E) = E \times \operatorname{Inndiag}(D_4(q))$. Furthermore, such an $E$ is unique up to conjugacy in $G_{\operatorname{ad}}$. 

It is known that $N_G(E)$ is always a maximal subgroup of $G$, and $N_{G_{\operatorname{ad}}}(E)$ is a maximal subgroup of $G_{\operatorname{ad}}$ unless $N_{G_{\operatorname{ad}}}(E) \leq G$. In this note, we describe the structure of $N_{G}(E)$. It turns out that $N_G(E) = N_{G_{\operatorname{ad}}}(E)$ if and only if $q \equiv \pm 1 \mod{8}$.
\end{abstract}

\maketitle

\section{Introduction} 

The purpose of this note is to clear up the structure of a certain maximal subgroup that appears in a finite simple group of type $E_7$. Specifically, the maximal subgroup that we consider arises as the normalizer of a certain elementary abelian $2$-subgroup of order $4$ in $E_7(q)$, where $q$ is an odd prime power. By results of Cohen--Liebeck--Saxl--Seitz \cite{CLSS1992}, the normalizer of this subgroup is known in the inner-diagonal automorphism group $\operatorname{Inndiag}(E_7(q))$, which contains $E_7(q)$ as a subgroup of index $2$. 

To start with, we first recall the standard setup for groups of Lie type, following \cite{SteinbergNotesAMS} and \cite[Chapter 2]{GLS3}, which we also refer to for more details. Let $\mathbf{G}$ be a simple algebraic group over $K = \overline{\mathbb{F}_p}$, where $p > 2$ is a prime. Let $q$ be a power of $p$, and let $\sigma$ be the corresponding standard Frobenius endomorphism, acting as $\sigma: x_{\alpha}(c) \mapsto x_{\alpha}(c^q)$ on the root elements of $\mathbf{G}$.

Throughout we will assume that $\mathbf{G}$ is simple of adjoint type $E_7$. Let $\mathbf{G}_{sc}$ be the simply connected cover of $\mathbf{G}$. Then $\mathbf{G}_{\sigma} = \operatorname{Inndiag}(E_7(q))$, with derived subgroup $\mathbf{G}_{\sigma}' = E_7(q)$ as an index two subgroup. Furthermore here $\left(\mathbf{G}_{sc} \right)_{\sigma} / Z(\mathbf{G}_{sc})_{\sigma} \cong \mathbf{G}_{\sigma}'$, where $Z(\mathbf{G}_{sc})_{\sigma} = Z(\mathbf{G}_{sc})$ is cyclic of order $2$.

As proven by Cohen, Liebeck, Saxl, and Seitz in \cite{CLSS1992}, there is an elementary abelian $2$-subgroup $E = 2^2$ in $\mathbf{G}$, such that $N_{\mathbf{G}}(E)/C_{\mathbf{G}}(E) \cong \operatorname{Sym}_3$, and $C_{\mathbf{G}}(E) = E \times D_4$, where the $D_4$ factor is of adjoint type. Furthermore, such an $E$ is unique up to conjugacy in $\mathbf{G}$. (Another way to characterize $E$ up to conjugacy in $\mathbf{G}$ is that every involution in $E \setminus \{1\}$ is of type $A_7$, see Proposition \ref{prop:Eproperties}  (vi).)

It is observed in \cite{CLSS1992} that there is a $\sigma$-invariant conjugate of $E$, unique up to conjugacy in $\mathbf{G}_{\sigma}$, such that \begin{align*}C_{\mathbf{G}_{\sigma}}(E) &= E \times \operatorname{Inndiag}(D_4(q)), \\ N_{\mathbf{G}_{\sigma}}(E)/C_{\mathbf{G}_{\sigma}}(E) &= \operatorname{Sym}_3.\end{align*} Then since $N_{\mathbf{G}_{\sigma}}(E)$ acts transitively on $E \setminus \{1\}$, we have $E \leq \mathbf{G}_{\sigma}'$. Because the lift of the $D_4$ factor in $\mathbf{G}_{sc}$ is also of adjoint type, the $\operatorname{Inndiag}(D_4(q)) = \operatorname{P\Omega}_8^{+}(q).2^2$ factor is also contained in $\mathbf{G}_{\sigma}'$ (Lemma \ref{lemma:sigmaprimebasic}). Thus one of the following holds:
	\begin{enumerate}[\normalfont (i)]
		\item $N_{\mathbf{G}_{\sigma}'}(E) = C_{\mathbf{G}_{\sigma}}(E).\operatorname{Sym}_3$.
		\item $N_{\mathbf{G}_{\sigma}'}(E) = C_{\mathbf{G}_{\sigma}}(E).3$
	\end{enumerate}

In this note we will prove the following result, which corrects one of the entries from \cite[Table 1]{AnDietrich2016} --- see Remark \ref{remark:andietrich} below.

\begin{lause}\label{thm:mainthm}
Let $E = 2^2$ be an elementary abelian $2$-subgroup of $\mathbf{G}$ such that $E_{\sigma} = E$, $C_{\mathbf{G}_{\sigma}}(E) = E \times \operatorname{Inndiag}(D_4(q))$, and $N_{\mathbf{G}_{\sigma}}(E)/C_{\mathbf{G}_{\sigma}}(E) = \operatorname{Sym}_3$. Then $$N_{\mathbf{G}_{\sigma}'}(E) = \begin{cases} C_{\mathbf{G}_{\sigma}}(E).\operatorname{Sym}_3, & \text{ if } q \equiv \pm 1 \mod{8}. \\ C_{\mathbf{G}_{\sigma}}(E).3, & \text{ if } q \equiv \pm 3 \mod{8}.\end{cases}$$ In particular $N_{\mathbf{G}_{\sigma}'}(E) = N_{\mathbf{G}_{\sigma}}(E)$ if and only if $q \equiv \pm 1 \mod{8}$.
\end{lause}

\begin{remark}\label{remark:maximality}
Let $X$ be an almost simple group with socle $L$ an exceptional group of Lie type. A \emph{local maximal subgroup} of $X$ is a maximal subgroup of the form $N_X(E)$, where $E$ is an elementary abelian $r$-group for some prime $r$. In \cite{CLSS1992}, Cohen--Liebeck--Saxl--Seitz prove that for a local maximal subgroup $M < X$ one of the following holds: (1) $M$ is parabolic; (2) $M$ is of maximal rank; (3) $(L,E)$ appears in \cite[Table 1]{CLSS1992}. 

In \cite{CLSS1992} the maximality of the local subgroups arising from $(L,E)$ in \cite[Table 1]{CLSS1992} was not proved (nor claimed). However, maximality is not too difficult to see from results that appeared later. For example, inclusion of $N_L(E)$ in an almost simple subgroup of $L$ can be ruled out by using \cite[Theorem 1]{LS1999angew}, \cite[Theorem 1]{LS1998trans}, and \cite[Theorem 2]{Lawther2014root}. 

To be clear, what is true is the following. For each pair $(L,E)$ in \cite[Table 1]{CLSS1992}, there exists $E \leq L$ which is uniquely determined up to conjugacy in $\operatorname{Inndiag}(L)$. Then $N_X(E)$ is a maximal subgroup of $X$ for all $L \leq X \leq N_{\Aut(L)}(E)L$. Note that if $X \not\leq N_{\Aut(L)}(E)L$, then $N_X(E)$ is contained in $X \cap N_{\Aut(L)}(E)L \lneqq X$ and is therefore not maximal.
\end{remark}

\begin{remark}\label{remark:andietrich}
It is claimed in \cite[Table 1]{AnDietrich2016} that $N_{\mathbf{G}_{\sigma}'}(E) = C_{\mathbf{G}_{\sigma}}(E).3$ holds in all cases, but as seen from Theorem \ref{thm:mainthm}, this equality holds only if $q \equiv \pm 3 \mod{8}$. This issue arises from the beginning of the proof of \cite[Lemma 6.3 f)]{AnDietrich2016}, which claims that \cite{CLSS1992} proves that $N_{\mathbf{G}_{\sigma}}(E)$ is maximal in $\mathbf{G}_{\sigma}$. However, as mentioned in Remark \ref{remark:maximality} above, no such statement is made in \cite{CLSS1992}. 

\end{remark}

\begin{acknow*}
Supported by NSFC grant 12350410360. The author would like to thank David Craven and Donna Testerman for helpful discussions and comments, and an anonymous referee for pointing out some misprints in an earlier version of this paper.
\end{acknow*}

\section{Notation and preliminaries}

\subsection{} Throughout the paper we use the notation as established in the introduction. Additionally, we use $I_d$ to denote a $d \times d$ identity matrix over $K$. If $A$ and $B$ are groups, then $G = A.B$ denotes that $G$ is an extension of $A$ and $B$, meaning that $G$ has a normal subgroup $N \cong A$ with quotient $G/N \cong B$. For an algebraic group $\mathbf{X}$, we denote by $\mathbf{X}^\circ$ the connected component of $X$.

\subsection{}\label{ss:involutions} We recall that in the adjoint group $\mathbf{G}$ of type $E_7$, there are three conjugacy classes of involutions. These are labelled by the structure of the centralizer, which for an involution in $\mathbf{G}$ has connected component of type $D_6A_1$, $E_6T_1$, or $A_7$. Another way to recognize the conjugacy class of an involution in $\mathbf{G}$ is by the dimension of its fixed point space on the adjoint module, as seen in the following table.

\begin{center}
\begin{tabular}{c|c}
Class     & Dimension of fixed point space on $\Lie(\mathbf{G})$ \\ \hline
$D_6A_1$  & $69$ \\
$E_6T_1$  & $79$ \\
$A_7$     & $63$
\end{tabular}
\end{center}

We also note that an involution of type $D_6A_1$ lifts to an involution in $\mathbf{G}_{sc}$, while involutions of type $E_6T_1$ and $A_7$ lift to an element of order $4$ in $\mathbf{G}_{sc}$.

For checking whether a given element $x \in \mathbf{G}_{\sigma}$ is contained in $\mathbf{G}_{\sigma}'$, we will typically use the following elementary observation.

\begin{lemma}\label{lemma:sigmaprimebasic}
Suppose that $x \in \mathbf{G}_{\sigma}$, and let $y \in \mathbf{G}_{sc}$ be a lift of $x$ into $\mathbf{G}_{sc}$. Then $x \in \mathbf{G}_{\sigma}'$ if and only if $y \in (\mathbf{G}_{sc})_{\sigma}$.
\end{lemma}

\begin{proof}
Follows from the fact that $\left(\mathbf{G}_{sc} \right)_{\sigma} / Z(\mathbf{G}_{sc})_{\sigma} \cong \mathbf{G}_{\sigma}'$, where the isomorphism is induced by an isogeny $\mathbf{G}_{sc} \rightarrow \mathbf{G}$.
\end{proof}

\section{Construction and uniqueness of \texorpdfstring{$E$}{E}}
 
\subsection{} 

For the proof of Theorem \ref{thm:mainthm}, we will first need an explicit construction of $E = \langle e,f \rangle$ in $\mathbf{G}$, which will be given in this section. We will also note that $E$ is characterized by the fact that every involution in $E \setminus \{1\}$ is of type $A_7$ (Proposition \ref{prop:Eproperties} (vi)).

\subsection{} 

Let $\alpha_1$, $\ldots$, $\alpha_7$ be the simple roots in the root system $\Phi$ of $E_7$, with standard Bourbaki labelling. 

For each $\alpha \in \Phi$ and $t \in K^{\times}$ recall the usual Chevalley generators \begin{align*} 
w_{\alpha}(t) & = x_{\alpha}(t) x_{-\alpha}(-t^{-1}) x_{\alpha}(t), \\ 
w_{\alpha} & = w_{\alpha}(1), \\ 
h_{\alpha}(t) & = w_{\alpha}(t) w_{\alpha}(1)^{-1}. 
\end{align*} Then we have a maximal torus $\mathbf{T}$ generated by $h_{\alpha_i}(c)$ with $1 \leq i \leq 7$ and $c \in K^\times$.

\subsection{} Let $\alpha_0 \in \Phi^+$ be the longest root. Then we have a standard subsystem subgroup $\mathbf{H}$ of type $A_7$, which corresponds to the subsystem of $\Phi$ with base $J = \{ -\alpha_0, \alpha_1, \alpha_3, \alpha_4, \alpha_5, \alpha_6, \alpha_7\}.$ Here $$\mathbf{H} = \langle \mathbf{T}, \mathbf{U}_{\alpha} : \alpha \in \Phi \cap \Z J \rangle,$$ where $\mathbf{U}_{\alpha} = \langle x_{\alpha}(t) : t \in K \rangle$ is the root subgroup corresponding to a root $\alpha$.

In a Chevalley group of type $A_n$ with simple roots $\{\beta_1, \ldots, \beta_n\}$ the center is generated by $\prod_{1 \leq i \leq n} h_{\beta_i}(c^i)$, where $c$ is a primitive $(n+1)$th root of unity. Thus if we let $\zeta \in K^\times$ be a primitive $8$th root of unity, the center of $\mathbf{H}$ is given by $$e = h_{-\alpha_0}(\zeta) h_{\alpha_1}(\zeta^2)h_{\alpha_3}(\zeta^3) \cdots h_{\alpha_7}(\zeta^7) = h_{\alpha_2}(-\zeta^2) h_{\alpha_5}(\zeta^2) h_{\alpha_7}(-\zeta^2) h_{\alpha_6}(-1).$$ We have $e^2 = h_{\alpha_2}(-1) h_{\alpha_5}(-1) h_{\alpha_7}(-1) = 1$, since the lift of $h_{\alpha_2}(-1) h_{\alpha_5}(-1) h_{\alpha_7}(-1)$ in $\mathbf{G}_{sc}$ generates $Z(\mathbf{G}_{sc})$.

It follows that $e$ is an involution and $\mathbf{H} \cong \SL_8(K) / \langle \zeta^2 I_8 \rangle$. Furthermore $\mathbf{H} = C_{\mathbf{G}}(e)^\circ$.

\subsection{}\label{ss:2p3} Next we define $f$ as an element of $\mathbf{G}$ which corresponds to the longest element of the Weyl group of $E_7$. Note that the longest element acts on $\Phi$ as $\alpha \mapsto -\alpha$. Let $$f = w_{\alpha_1} w_{\alpha_2} w_{\alpha_5} w_{\alpha_7} w_{\alpha_{37}} w_{\alpha_{55}} w_{\alpha_{61}},$$ where for $i > 7$, we denote by $\alpha_i$ the $i$-th positive root in $\Phi$, with respect to the ordering of roots of $E_7$ used by {\sc Magma} \cite{MAGMA}. Here specifically $\alpha_{37} = 1122100$, $\alpha_{55} = 1122221$, $\alpha_{61} = 1224321$, where $k_1k_2 \dots k_7$ denotes $\sum_{1 \leq i \leq 7} k_i \alpha_i$.

Now $f$ is a product of pairwise commuting Weyl group representatives $w_{\alpha}$. Since $w_{\alpha}^2 = h_{\alpha}(-1)$ for all $\alpha \in \Phi$, it is easily seen that $$f^2 = h_{\alpha_2}(-1)h_{\alpha_5}(-1)h_{\alpha_7}(-1) = 1,$$ so $f$ is an involution. 

Since $f$ acts as $-1$ on $\Phi$, it inverts the torus $\mathbf{T}$. Therefore $$fef^{-1} = h_{\alpha_2}(-1) h_{\alpha_5}(-1) h_{\alpha_7}(-1) e = e,$$ so $f$ centralizes $e$. We also see from the same calculation in $\mathbf{G}_{sc}$ that the lifts of $f$ and $e$ to $\mathbf{G}_{sc}$ do not commute, so $f \in C_{\mathbf{G}}(e) \setminus C_{\mathbf{G}}(e)^\circ$. Consequently $C_{\mathbf{G}}(e) = \mathbf{H} \rtimes \langle f \rangle$.

We also note that $$f x_{\alpha}(c) f^{-1} = x_{-\alpha}(-c)$$ for all $\alpha \in \Phi$ and $c \in K^\times$, so $f$ acts on $\mathbf{H} = C_{\mathbf{G}}(e)^\circ = \SL_8(K) / \langle \zeta^2 I_8 \rangle$ as the inverse-transpose automorphism. 

\subsection{} Now $$E := \langle e,f \rangle$$ is conjugate in $\mathbf{G}$ to the $2$-subgroup of $\mathbf{G}_{\sigma}$ that we are interested in. This is seen by the following proposition, which is for the most part established in the proof of \cite[Lemma 2.15]{CLSS1992}.

\begin{prop}\label{prop:Eproperties} The following hold:
	\begin{enumerate}[\normalfont (i)]
		\item Every involution in $E$ is of type $A_7$.
		\item $C_{\mathbf{G}}(E) = E \times D_4$, where the $D_4$ factor is of adjoint type.
		\item $N_{\mathbf{G}}(E)/C_{\mathbf{G}}(E) \cong \Sym_3$.
		\item $N_{\mathbf{G}}(E)/C_{\mathbf{G}}(E)$ acts faithfully on $E$, and by graph automorphisms on the $D_4$ factor in $C_{\mathbf{G}}(E)$.
		\item $N_{\mathbf{G}}(E)/N_{\mathbf{G}}(E)^\circ \cong \Sym_4$.
		\item Suppose that $F \leq \mathbf{G}$ is elementary abelian of order $4$ such that every element of $F \setminus \{1\}$ is an involution of type $A_7$. Then $F$ is $\mathbf{G}$-conjugate to $E$.
	\end{enumerate}
\end{prop}

\begin{proof}
For (i), by construction $C_{\mathbf{G}}(e)^\circ$ is a subsystem subgroup of type $A_7$, so $e$ is an involution of type $A_7$. We have seen in \ref{ss:2p3} that $f$ inverts the maximal torus $\mathbf{T}$ and swaps the positive and negative root spaces on $\Lie(\mathbf{G})$. Thus the fixed point space of $f$ on $\Lie(\mathbf{G})$ has dimension $|\Phi|/2 = 63$, which by \ref{ss:involutions} implies that $f$ is an involution of type $A_7$. For the element $ef$, take $g \in \mathbf{T}$ to be an element in the maximal torus such that $g^2 = e$. Since $f$ inverts the maximal torus, we have $gfg^{-1} = g^2f = ef$, so $f$ is conjugate to $ef$.

For (ii), recall from \ref{ss:2p3} that $f$ acts on $C_{\mathbf{G}}(e)^\circ = \SL_8(K) / \langle \zeta^2 I_8 \rangle$ by the inverse-transpose graph automorphism. Then as observed in \cite[Lemma 2.15]{CLSS1992}, we have $C_{\mathbf{G}}(E) = E \times D_4$. Here the $D_4$ factor is the image of $\SO_8(K)$ in $C_{\mathbf{G}}(e)^\circ = \SL_8(K) / \langle \zeta^2 I_8 \rangle$, and is thus of adjoint type. 

Claim (iii) and (iv) follow from the argument in the first paragraph of \cite[p. 35]{CLSS1992}. For (v), by (ii) -- (iv) we have $N_{\mathbf{G}}(E)/N_{\mathbf{G}}(E)^\circ = E.\Sym_3$, with $\Sym_3$ acting faithfully on $E$. The only group of order $24$ with these properties is $\Sym_4 \cong E \rtimes \Sym_3$.

For claim (vi), let $F \leq \mathbf{G}$ such that $|F| = 4$ and every element of $F \setminus \{1\}$ is an involution of type $A_7$. By replacing $F$ with a conjugate, we can assume that $e \in F$, so $F = \langle e,f' \rangle$ for some involution $f'$ of type $A_7$.

Suppose first that $f' \in C_{\mathbf{G}}(e)^\circ = \SL_8(K) / \langle \zeta^2 \rangle$. Since $f'$ is an involution, up to conjugacy it is the image of a diagonal matrix $\diag(\lambda I_a, -\lambda I_{8-a})$ for some $\lambda \in \langle \zeta \rangle$. Hence $f'$ is the image of a diagonal matrix of the form $\diag(\zeta I_a, -\zeta I_{8-a})$ or $\diag(I_a, -I_{8-a})$. Then either $ef'$ or $f'$ lifts to an involution in $\mathbf{G}_{sc}$, but this cannot happen for an involution of type $A_7$, as noted in \ref{ss:involutions}.

Therefore we must have $f' \in C_{\mathbf{G}}(e) \setminus C_{\mathbf{G}}(e)^\circ$. By \cite[Proposition 2.7]{CLSS1992}, there are two $C_{\mathbf{G}}(e)^\circ$-classes of involutions in $C_{\mathbf{G}}(e) \setminus C_{\mathbf{G}}(e)^\circ$. First, we already know by (i) that $f \in C_{\mathbf{G}}(e) \setminus C_{\mathbf{G}}(e)^\circ$ is an involution of type $A_7$ in $\mathbf{G}$. Moreover, by the proof of \cite[Lemma 2.15]{CLSS1992} (third paragraph), there exists an involution of type $E_6T_1$ in $C_{\mathbf{G}}(e) \setminus C_{\mathbf{G}}(e)^\circ$. Thus $f'$ must be $C_{\mathbf{G}}(e)^\circ$-conjugate to $f$, which proves that $F$ is $\mathbf{G}$-conjugate to $E$.
\end{proof}

\begin{remark}
Proposition \ref{prop:Eproperties} (vi) is also a consequence of a result of Griess \cite[Proposition 9.5]{Griess}.
\end{remark}

\section{On \texorpdfstring{$\sigma$}{σ}-invariant conjugates of \texorpdfstring{$E$}{E}}

\subsection{} We have $\sigma(e) = e$ and $\sigma(f) = f$, so $E$ is $\sigma$-invariant and $E_{\sigma} = E$. In this section, we will consider the possibilities for $\sigma$-invariant conjugates of $E$. 

\subsection{}\label{ss:langsteinberg} Let $X$ be a group on which $\sigma$ acts. We denote by $H^1(\sigma, X)$ the equivalence classes of $X$ under the relation $\sim$ defined by $x \sim y$ if and only if $x = \sigma(g)^{-1}yg$ for some $g \in X$. 

Denote by $\mathcal{E}$ the set of $\sigma$-invariant $\mathbf{G}$-conjugates of $E$, and let $\mathcal{E}/\mathbf{G}_{\sigma}$ be the set of $\mathbf{G}_{\sigma}$-conjugacy classes in $\mathcal{E}$. Then by \cite[2.7]{SpringerSteinberg}, we have a bijection $$\mathcal{E}/\mathbf{G}_{\sigma} \rightarrow H^1(\sigma, N_{\mathbf{G}}(E)/N_{\mathbf{G}}(E)^\circ)$$ which maps $[\mathbf{X}^g]$ to the image of $\sigma(g)g^{-1}$ in $H^1(\sigma, N_{\mathbf{G}}(E)/N_{\mathbf{G}}(E)^\circ)$.

\subsection{} Thus the $\mathbf{G}_{\sigma}$-classes of $\sigma$-invariant conjugates of $E$ are parametrized by $H^1(\sigma, N_{\mathbf{G}}(E) / N_{\mathbf{G}}(E)^\circ)$, where $$N_{\mathbf{G}}(E) / N_{\mathbf{G}}(E)^\circ \cong E \rtimes \operatorname{Sym}_3 \cong \operatorname{Sym}_4$$ by Proposition \ref{prop:Eproperties} (v). 

\subsection{}\label{ss:2p6} Now $\sigma$ acts on $N_{\mathbf{G}}(E) / N_{\mathbf{G}}(E)^\circ$ with conjugation by some element of $N_{\mathbf{G}}(E)$, since every automorphism of $\operatorname{Sym}_4$ is inner. Note that $\sigma$ acts trivially on $E = E_{\sigma}$, so the action of $\sigma$ on $N_{\mathbf{G}}(E) / N_{\mathbf{G}}(E)^\circ$ is given with conjugation by some element $e' \in E$. 

Then the elements of $E$ map into two classes in $H^1(\sigma, N_{\mathbf{G}}(E)/N_{\mathbf{G}}(E)^\circ)$, corresponding to $\{e'\}$ and $E \setminus\{e'\}$. We will see that $e' = 1$ or $e' = f$, according to whether $q \equiv 1 \mod{4}$ or $q \equiv 3 \mod{4}$.

\subsection{} We have $f \in C_{\mathbf{G}}(f)^\circ$, so by the Lang-Steinberg theorem we can write $f = \sigma(x)x^{-1}$ for some $x \in C_{\mathbf{G}}(f)^\circ$. Then the $\sigma$-invariant conjugate of $E$ corresponding to $f$ (as in \ref{ss:langsteinberg}) is $$\widetilde{E} := x^{-1} E x = \langle \widetilde{e}, f \rangle,$$ where $\widetilde{e} = x^{-1} e x$. Note that $\widetilde{E}_{\sigma} = \widetilde{E}$.

For later use, we make two observations about the actions of $\sigma$ and $f$ on the centralizers $C_{\mathbf{G}}(e)^\circ$ and $C_{\mathbf{G}}(\widetilde{e})^\circ$.

\begin{lemma}\label{lemma:torusactions}
We have the following:
	\begin{enumerate}[\normalfont (i)]
		\item $\sigma$ acts on the maximal torus $\mathbf{T}$ of $C_{\mathbf{G}}(e)^\circ$ via $t \mapsto t^q$.
		\item $\sigma$ acts on the maximal torus $x^{-1} \mathbf{T} x$ of $C_{\mathbf{G}}(\widetilde{e})^\circ$ via $t \mapsto t^{-q}$.
	\end{enumerate}
\end{lemma}

\begin{proof}
Here (i) holds since $\sigma(h_{\alpha}(c)) = h_{\alpha}(c^q) = h_{\alpha}(c)^q$, and (ii) holds since $x$ inverts $\mathbf{T}$.
\end{proof}

\begin{lemma}\label{lemma:faction}
We have the following:
	\begin{enumerate}[\normalfont (i)]
		\item Identifying $\mathbf{T}$ as the maximal torus of $C_{\mathbf{G}}(e)^\circ$ corresponding to diagonal matrices, $f$ acts on $C_{\mathbf{G}}(e)^\circ$ as the inverse-transpose automorphism.
		\item Identifying $x^{-1}\mathbf{T}x$ as the maximal torus of $C_{\mathbf{G}}(\widetilde{e})^\circ$ corresponding to diagonal matrices, $f$ acts on $C_{\mathbf{G}}(\widetilde{e})^\circ$ as the inverse-transpose automorphism.
	\end{enumerate}
\end{lemma}

\begin{proof}
Here (i) was noted in \ref{ss:2p3}. Then (ii) follows from (i), since $x$ centralizes $f$.
\end{proof}

\subsection{} We will now see that $E$ and $\widetilde{E}$ are not $\mathbf{G}_{\sigma}$-conjugate. For this we first need the following lemma.

\begin{lemma}\label{lemma:sigmaprime}
We have the following:
	\begin{enumerate}[\normalfont (i)]
		\item $E_{\sigma} \leq \mathbf{G}_{\sigma}'$ if and only if $q \equiv 1 \mod{4}$. 
		\item $\widetilde{E}_{\sigma} \leq \mathbf{G}_{\sigma}'$ if and only if $q \equiv 3 \mod{4}$. 
	\end{enumerate}
\end{lemma}

\begin{proof}
First note that $f$ lifts to an element of $\mathbf{G}_{sc}$ fixed by $\sigma$ (since $w_{\alpha}$ is always fixed by $\sigma$), so by Lemma \ref{lemma:sigmaprimebasic} we have $f \in \mathbf{G}_{\sigma}'$. Thus for containment in $\mathbf{G}_{\sigma}'$, we just need to look at the other generators $e$ and $\widetilde{e}$.

The lift of $e$ is an element $e_1$ of $\mathbf{G}_{sc}$ such that $e_1$ has order $4$, and $\sigma(e_1) = e_1^q$ (Lemma \ref{lemma:torusactions}). We have $e_1^q = e_1$ if and only if $q \equiv 1 \mod{4}$, so (i) holds (Lemma \ref{lemma:sigmaprimebasic}).

Similarly note that the lift of $\widetilde{e}$ is an element $e_2$ of $\mathbf{G}_{sc}$ such that $e_2$ has order $4$, and $\sigma(e_2) = e_2^{-q}$ (Lemma \ref{lemma:torusactions}). We have $e_2^{-q} = e_2$ if and only if $q \equiv 3 \mod{4}$, so (ii) holds (Lemma \ref{lemma:sigmaprimebasic}).
\end{proof}

\subsection{} It follows from Lemma \ref{lemma:sigmaprime} that $E$ and $\widetilde{E}$ are not conjugate in $\mathbf{G}_{\sigma}$. (Since one is contained in $\mathbf{G}_{\sigma}'$, and the other is not.) Thus the identity element $1$ and $f$ correspond to different $\sigma$-classes in $H^1(\sigma, N_{\mathbf{G}}(E) / N_{\mathbf{G}}(E)^\circ)$, so from the discussion in \ref{ss:2p6} we conclude that $\sigma$ acts on $N_{\mathbf{G}}(E) / N_{\mathbf{G}}(E)^\circ$ as conjugation by $f$ or as conjugation by $1$.

\subsection{} We are interested in the $\sigma$-invariant conjugate of $E$ such that the structure of $N_{\mathbf{G}_{\sigma}}(E)$ is as stated in Theorem \ref{thm:mainthm}. For this, among $E$ and $\widetilde{E}$ the right conjugate is the one contained in $\mathbf{G}_{\sigma}'$. Indeed, if for example $E$ is not contained in $\mathbf{G}_{\sigma}'$, then $N_{\mathbf{G}_{\sigma}}(E)$ normalizes $E \cap \mathbf{G}_{\sigma}' = \langle f \rangle$, and thus centralizes $f$. In this case $N_{\mathbf{G}_{\sigma}}(E)$ is contained in a subgroup of type $A_7$, which is not the case for the subgroup appearing in Theorem \ref{thm:mainthm}.

\subsection{}\label{ss:sigmaaction} Let $g \in C_{\mathbf{G}}(e)^\circ$ be an element corresponding to the diagonal matrix $$\operatorname{diag}(\sqrt{\zeta}, \sqrt{\zeta}, \ldots, \sqrt{\zeta}, -\sqrt{\zeta}).$$ Then $g \in \mathbf{T}$, and $g^2 = e$. Moreover $g$ is inverted by $f$ (Lemma \ref{lemma:faction} (i)), so \begin{align*} geg^{-1} &= e, \\ gfg^{-1} &= ef, \\ g(ef)g^{-1} &= f. \end{align*} Thus $g \in N_{\mathbf{G}}(E)$ and $g$ corresponds to an involution in $N_{\mathbf{G}}(E)/C_{\mathbf{G}}(E) \cong \operatorname{Sym}_3$ that swaps $f$ and $ef$.

If $q \equiv 1 \mod{4}$, then $$\sigma(g)^{-1} f g = g^{-q} f g = g^{-1} f g = ef.$$ Thus $f$ and $ef$ correspond to the same $\sigma$-class in $H^1(\sigma, N_{\mathbf{G}}(E) / N_{\mathbf{G}}(E)^\circ)$, and so $\sigma$ acts trivially on $N_{\mathbf{G}}(E) / N_{\mathbf{G}}(E)^\circ$ in this case.

If $q \equiv 3 \mod{4}$, then $$\sigma(g)^{-1}g = g^{-q+1} = g^2 = e.$$ Then $1$ and $e$ correspond to the same $\sigma$-class in $H^1(\sigma, N_{\mathbf{G}}(E) / N_{\mathbf{G}}(E)^\circ)$. Therefore $\sigma$ acts on $N_{\mathbf{G}}(E) / N_{\mathbf{G}}(E)^\circ$ by conjugation with $f$, and thus $\sigma$ acts trivially on $N_{\mathbf{G}}(\widetilde{E}) / N_{\mathbf{G}}(\widetilde{E})^\circ$.

\subsection{} Now define $E^+ = E$, and $E^{-} = \widetilde{E}$. To summarize, first we have the following result.

\begin{prop}\label{prop:Esigmastructure}
Let $\varepsilon = \pm$ such that $q \equiv \varepsilon \mod{4}$. Then the following hold.

	\begin{enumerate}[\normalfont (i)]
		\item $\sigma$ acts trivially on $N_{\mathbf{G}}(E^{\varepsilon})/N_{\mathbf{G}}(E^{\varepsilon})^\circ$.
		\item $C_{\mathbf{G}_{\sigma}}(E^\varepsilon) = E^\varepsilon \times \operatorname{Inndiag}(D_4(q))$.
		\item $N_{\mathbf{G}_{\sigma}}(E^\varepsilon) / C_{\mathbf{G}_{\sigma}}(E^\varepsilon) = \operatorname{Sym}_3$.
	\end{enumerate}

\end{prop}

\begin{proof}
Claim (i) was established in \ref{ss:sigmaaction}. For (ii) and (iii), note that $$\left(N_{\mathbf{G}}(\tilde{E})^\circ\right)_{\sigma} = x \left(N_{\mathbf{G}}(E)^\circ\right)_{\sigma} x^{-1}$$ since $f = \sigma(x)x^{-1}$ centralizes $N_{\mathbf{G}}(E)^\circ = D_4$. Moreover $N_{\mathbf{G}}(E)^\circ = D_4$ is of adjoint type (Proposition \ref{prop:Eproperties} (ii)), so $$\left(N_{\mathbf{G}}(E^{\varepsilon})^\circ\right)_{\sigma} = \operatorname{Inndiag}(D_4(q)).$$ By (i) and Proposition \ref{prop:Eproperties} (v) we have $$N_{\mathbf{G}}(E^{\varepsilon})_{\sigma}/\left(N_{\mathbf{G}}(E^{\varepsilon})^\circ\right)_{\sigma} \cong \left(N_{\mathbf{G}}(E^{\varepsilon})/N_{\mathbf{G}}(E^{\varepsilon})^\circ\right)_{\sigma} \cong E^{\varepsilon} \rtimes \Sym_3 \cong \Sym_4.$$ (Here the first isomorphism holds for example by \cite[Proposition 23.2]{MalleTesterman}.) 

Therefore $N_{\mathbf{G}_{\sigma}}(E^\varepsilon) = \operatorname{Inndiag}(D_4(q)).\Sym_4$, from which claims (ii) and (iii) follow.
\end{proof}

In Table \ref{table:sigma}, we have given the structure of $N_{\mathbf{G}}(F)_{\sigma}$ for all $\sigma$-invariant conjugates $F$ of $E$. In the table, we identify $N_{\mathbf{G}}(E^{\varepsilon})/N_{\mathbf{G}}(E^{\varepsilon})^\circ = \Sym_4 = E^{\varepsilon} \rtimes \Sym_3$.

The structures in Table \ref{table:sigma} are found similarly to Proposition \ref{prop:Esigmastructure}. For example, in the case of $[(1,2)]$, the action of $\sigma$ on $N_{\mathbf{G}}(F) = D_4$ is by a graph automorphism of order $2$, so $\left(N_{\mathbf{G}}(F)^\circ\right)_{\sigma} = \operatorname{Inndiag}({}^2D_4(q)) = {}^2D_4(q).2$. 

Furthermore the centralizer of $(1,2)$ in $\Sym_4$ is elementary abelian of order $4$, so $N_{\mathbf{G}}(F)_{\sigma}/\left(N_{\mathbf{G}}(F)^\circ\right)_{\sigma} = 2^2$. There is a unique involution in $F$ fixed by $\sigma$, so we conclude that $$N_{\mathbf{G}}(F)_{\sigma} = \left(2 \times {}^2D_4(q).2\right)\!.2$$ holds.

\begin{remark}
As noted by Craven in \cite[Section 4, pp. 15--16]{CravenE7}, the subgroups in Table \ref{table:sigma} corresponding to $[(1,2)]$, $[(1,2)(3,4)]$, and $[(1,2,3,4)]$ do not produce maximal subgroups of $E_7(q)$, since all of them centralize an involution in $F$. Meanwhile $[(1,2,3)]$ corresponds to maximal subgroup of $E_7(q)$ with structure ${}^3D_4(q).3$, as observed in \cite[Section 4, p. 16]{CravenE7}.
\end{remark}

\begin{table}[!htbp]
	\centering
	\caption{Structure of $N_{\mathbf{G}}(F)_{\sigma}$ for $\sigma$-invariant conjugates of $E$. In the first column we denote the element of $H^1(\sigma,N_{\mathbf{G}}(E^{\varepsilon})/N_{\mathbf{G}}(E^{\varepsilon})^\circ)$ corresponding to $F$, where we have identified $N_{\mathbf{G}}(E^{\varepsilon})/N_{\mathbf{G}}(E^{\varepsilon})^\circ = \Sym_4$.}\label{table:sigma}
	\begin{tabular}{|l|l|}
		\hline
		& \\[-9pt]
		Class in $H^1(\sigma,N_{\mathbf{G}}(E^{\varepsilon})/N_{\mathbf{G}}(E^{\varepsilon})^\circ)$ & Structure of $N_{\mathbf{G}}(F)_{\sigma}$ \\[2pt] \hline
		& \\[-9pt]
		$[(1)]$      & $\left(2^2 \times \operatorname{Inndiag}(D_4(q))\right)\!.\Sym_3$ \\
		$[(1,2)]$     & $\left(2 \times {}^2D_4(q).2\right)\!.2$ \\
		$[(1,2)(3,4)]$ & $\left(2^2 \times \operatorname{Inndiag}(D_4(q))\right)\!.2$ \\
		$[(1,2,3)]$    & ${}^3D_4(q).3$ \\
		$[(1,2,3,4)]$   & $\left({}^2D_4(q).2\right)\!.4$ \\[2pt]
    \hline
	\end{tabular}
\end{table}

\section{Proof of Theorem \ref{thm:mainthm}}

We will now prove Theorem \ref{thm:mainthm}, which is a consequence of the following result.

\begin{lause}
Let $\varepsilon = \pm$ such that $q \equiv \varepsilon \mod{4}$. Then the following hold:
	\begin{enumerate}[\normalfont (i)]
	\item $C_{\mathbf{G}_{\sigma}}(E^{\varepsilon}) \leq \mathbf{G}_{\sigma}'$.
	\item $$N_{\mathbf{G}_{\sigma}'}(E^{\varepsilon}) = \begin{cases} C_{\mathbf{G}_{\sigma}}(E^{\varepsilon}).\operatorname{Sym}_3, & \text{ if } q \equiv \varepsilon \mod{8}. \\ C_{\mathbf{G}_{\sigma}}(E^{\varepsilon}).3, & \text{ if } q \equiv \varepsilon+4 \mod{8}.\end{cases}$$
	\end{enumerate}
\end{lause}

\begin{proof}
For claim (i), first note that $E^{\varepsilon} \leq \mathbf{G}_{\sigma}'$ by Lemma \ref{lemma:sigmaprime}. Next we consider the $Y = D_4$ factor in $C_{\mathbf{G}}(E^\varepsilon) = E^{\varepsilon} \times Y$, which is of adjoint type. The lift of $Y$ into $\mathbf{G}_{sc}$ is also of adjoint type, from which it follows that $Y_{\sigma} \leq \mathbf{G}_{\sigma}'$ (Lemma \ref{lemma:sigmaprimebasic}). Thus $C_{\mathbf{G}_{\sigma}}(E^{\varepsilon}) = E^{\varepsilon} \times Y_{\sigma} \leq \mathbf{G}_{\sigma}'$.

For (ii), let $g \in \mathbf{T}$ be the element of $N_{\mathbf{G}}(E)$ defined in \ref{ss:sigmaaction}, so $g^2 = e$ and $g$ acts on $E$ by swapping $f$ and $ef$. Define $$y := \begin{cases} g, & \text{ if } q \equiv 1 \mod{4}, \\ x^{-1} g x, & \text{ if } q \equiv 3 \mod{4}.\end{cases}$$ Then $y \in N_{\mathbf{G}}(E^{\varepsilon})$. By Lemma \ref{lemma:torusactions} we have $\sigma(y) = y^{q \varepsilon} = y$ since $y$ has order $4$ and $q \varepsilon \equiv 1 \mod{4}$. Thus $y \in N_{\mathbf{G}_{\sigma}}(E^{\varepsilon})$. 

Combining this with $C_{\mathbf{G}_{\sigma}}(E^{\varepsilon}) \leq \mathbf{G}_{\sigma}'$, it follows that $$N_{\mathbf{G}_{\sigma}}(E^{\varepsilon}) = \langle N_{\mathbf{G}_{\sigma}'}(E^{\varepsilon}), y \rangle.$$ Then claim (ii) will follow once we prove that $y \in \mathbf{G}_{\sigma}'$ if and only if $q \equiv \varepsilon \mod{8}$. To this end, let $y'$ be a lift of $y$ into $\mathbf{G}_{sc}$. Then $y'$ has order $8$, so by Lemma \ref{lemma:torusactions} we have $$\sigma(y') = (y')^{\varepsilon q} = \begin{cases} y', & \text{ if } q \equiv \varepsilon \mod{8}, \\ (y')^5, & \text{ if } q \equiv \varepsilon+4 \mod{8}.\end{cases}$$ Thus by Lemma \ref{lemma:sigmaprimebasic}, we have $y \in \mathbf{G}_{\sigma}'$ if and only if $q \equiv \varepsilon \mod{8}$.
\end{proof}


\end{document}